\documentclass[12pt]{article}

\usepackage{amsmath}
\usepackage{amssymb}
\usepackage{amsthm}
\usepackage{bbm}
\usepackage{graphicx}
\usepackage{color}
\usepackage{epsfig}

\newtheorem{theorem}{Theorem}

\newtheorem{prop}[theorem]{Proposition}
\newtheorem{lemma}[theorem]{Lemma}
\newtheorem{cor}[theorem]{Corollary}

\newtheorem{question}{Question}

\newtheorem{conj}[question]{Conjecture}

\newcommand{\wt}{\ensuremath{\textrm{wt}}}

\newcommand{\floor}[1]{\ensuremath{\left \lfloor {#1} \right \rfloor}}

\newcommand{\ang}[1]{\ensuremath{\left \langle {#1} \right \rangle}}
\newcommand{\N}{{\mathbb N}}
\newcommand{\Z}{{\mathbb Z}}

\newcommand{\R}{{\mathbb R}}

\newcommand{\1}{{\mathbbmss 1}}

\newcommand{\be}{{\bf e}}
\newcommand{\Bf}{{\bf f}}

\newcommand{\bE}{{\bf E}}
\newcommand{\bG}{{\bf G}}

\newcommand{\bu}{{\bf u}}
\newcommand{\bP}{{\bf P}}

\newcommand{\br}{{\bf r}}
\newcommand{\bv}{{\bf v}}

\newcommand{\bx}{{\bf x}}

\newcommand{\cE}{{\mathcal E}}

\newcommand{\cG}{{\mathcal G}}
\newcommand{\cH}{{\mathcal H}}

\newcommand{\cL}{{\mathcal L}}

\newcommand{\cS}{{\mathcal S}}

\newcommand{\tr}{\textrm{tr}}

\DeclareMathOperator{\supp}{supp}

\DeclareMathOperator{\vol}{vol}
\DeclareMathOperator{\diag}{diag}
\DeclareMathOperator{\rank}{rank}
\DeclareMathOperator{\nl}{null}

\allowdisplaybreaks

\title{Deducing Vertex Weights from Empirical Occupation Times}
\author{Joshua Cooper}

\begin{document}

\maketitle

\begin{abstract} We consider the following problem arising from the study of human problem solving: Let $G$ be a vertex-weighted graph with marked ``in'' and ``out'' vertices.  Suppose a random walker begins at the in-vertex, steps to neighbors of vertices with probability proportional to their weights, and stops upon reaching the out-vertex.  Could one deduce the weights from the paths that many such walkers take?  We analyze an iterative numerical solution to this reconstruction problem, in particular, given the empirical mean occupation times of the walkers.  In the process, a result concerning the differentiation of a matrix pseudoinverse is given, which may be of independent interest.  We then consider the existence of a choice of weights for the given occupation times, formulating a natural conjecture to the effect that -- barring obvious obstructions -- a solution always exists.  It is shown that the conjecture holds for a class of graphs that includes all trees and complete graphs.  Several open problems are discussed.
\end{abstract}

\section{Introduction}

Single-agent search problems are commonly modeled as a graph $G$, with an edge from $x \in V(G)$ to $y \in V(G)$ (i.e., $x \sim y$) if it is possible to move from state $x$ to state $y$.  We will assume throughout that such ``moves'' $x \rightarrow y$ are reversible, so that $G$ is an undirected graph.  One particular vertex $\bv_{\textrm{out}}$ is the ``finish'' and another vertex $\bv_{\textrm{in}}$ is the ``start.''  The former is intended to model the solution of the problem being considered, and the latter the initial state of the solver.

Typical examples of such single-agent search problems include:
\begin{enumerate}
\item Vertices are states of a Rubik's Cube or a 15-puzzle, with an edge between two vertices if it is possible to transform one into the other by a standard move.  Here $\bv_{\textrm{in}}$ is the starting state (perhaps the result of a random walk in $G$) and $\bv_{\textrm{out}}$ is the solved puzzle.
\item Vertices are web pages, with edges corresponding to hyperlinks.  In this case, $\bv_{\textrm{in}}$ may be a company homepage, and $\bv_{\textrm{out}}$ a page where purchases are made (the ``check-out'').
\item Vertices are the positions of a chess board, edges correspond to legal moves by one player (perhaps a computer), $\bv_{\textrm{in}}$ is the initial position given by a chess puzzle, and $\bv_{\textrm{out}}$ is the set of all winning configurations (checkmates, captures, etc.).
\item The vertices are a grid of points in a mouse maze, with edges corresponding to feasible moves (i.e., missing walls); $\bv_{\textrm{in}}$ is the cage door and $\bv_{\textrm{out}}$ is the cheese.
\end{enumerate}

In many such examples, a researcher has access to the state of the solver, but not to their reasoning process (their ``policy'' to use machine learning parlance).  The amount of time a subject takes to find the solution state (the ``latency'') can serve as a useful proxy for their knowledge level, but this single number is a somewhat crude measurement.  One might strive to learn in addition the value attributed by the solver to intermediate states, i.e., the solver's ``value function.''   Such detailed profiles of preferences could aide in, for example, improving customer service, evaluating individual expertise, estimating how well a lab animal has learned a task, tuning a software game-playing engine, or identifying gaps in students' knowledge.  However, the solver -- human, lab animal, machine -- may be long gone, may not have conscious knowledge of this information, may be secretive, or may not be able to express their thoughts in a human-readable format.  Nonetheless, by studying the path that many instances of the solver take, one could hope to reconstruct such valuational ascriptions without the involvement of the solvers.  This strategy is akin to using the density of oil stains in a parking lot to see which spots are most popular or classifying historical road use by the depth of wheel-ruts.

We model the solution process as a random walk on the graph $G$, starting at $\bv_{\textrm{in}}$ and ending at $\bv_{\textrm{out}}$.  Vertex weights specify the proportional probabilities of moves, and encode the aforementioned value function.  A novice solver is presumed to follow a uniform random walk; that is, the transition probability to go from a state to one of its neighbors is the same for each neighbor.  The expert solver follows a more direct route from start to finish, as they are inclined to move closer to the solution state with each move.

In the next section, we describe our model in greater detail and relate the vertex weights to empirical mean occupation times.  The following section relates an iterative algorithm for the numerical solution of the problem of determining the weights from occupation times.  The analysis requires differentiation of a matrix pseudoinverse, something which may have independent interest.  Next, we discuss the matter of solution existence: When is it possible in principle to reconstruct the vertex weights?  We formulate a natural conjecture and prove that it holds for a class of graphs that includes all trees and complete graphs.  The final section discusses several open problems that have arisen in this context.

\section{The Model}

For each pair of vertices $x,y \in V$, we denote by $d(x,y)$ the graphical distance between $x$ and $y$, i.e., the length of the shortest path that begins at $x$ and ends at $y$.  $N(x)$ denotes the ``neighborhood'' of $x$, i.e., the set of all vertices adjacent to $v$.  The quantity $\deg(v)$, the ``degree'' of $v$, refers to the number of edges incident to $v \in V$.

Let $\rho:V \rightarrow \R^{\geq 0}$ be nonincreasing in distance from $\bv_{\textrm{out}}$, i.e., $d(x,\bv_{\textrm{out}}) \geq d(z,\bv_{\textrm{out}}) \Rightarrow \rho(x) \leq \rho(z)$ for each $x,z \in V$.  Then we define $P(x,y)$, the probability that the solver transitions from state $x$ to state $y$ by
$$
P(x,y) = \frac{\rho(y)}{\sum_{z \in N(x)} \rho(z)}.
$$
We also use the notation $P^{(t)}(x,y)$ to mean the probability that a walker starting from $x$ arrives at $y$ after exactly $t$ steps.  Such a distribution corresponds precisely to a reversible Markov chain, starting from $\bv_{\textrm{in}}$ and halted at $\bv_{\textrm{out}}$.

Given such a group of solvers, we have an empirical mean ``occupation time'' for each vertex $V$, given by
$$
\hat{\tau}(v) = N^{-1} \sum_{k=1}^N \tau_k(v),
$$
where $\tau_k(v)$ is the number of visits to site $v$ that subject $k$ makes before arriving at $\bv_{\textrm{out}}$.  It would be useful to understand how $\rho$ relates to $\tau$.

Suppose that we perform a random walk on $G$ according to the distribution $P$ arising from $\rho$ as above, starting at $\bv_{\textrm{in}}$ and stopping at the hitting time of $\bv_{\textrm{out}}$.  Note that the random walk arising from $P$ is also the random walk one gets by taking edge weights $\wt(x,y) = \rho(x) \rho(y)$, since the ratio of weights of neighbors of a point is the same.  If we define $\tilde{\rho}(x) = \sum_{y \sim x} \wt(x,y) = \rho(x) \sum_{y \sim x} \rho(y)$, then the corresponding stationary distribution at the point $x$ is $\tilde{\rho}(x)/\vol(G)$, where
$$
\vol(G) = \sum_{y \in V} \tilde{\rho}(y) = 2 \sum_{\{y,z\} \in E(G)} \wt(y,z).
$$

Applying \cite{AF}, Chapter 2, Lemma 9, we have
$$
\bE(\tau(x)) = \frac{\tilde{\rho}(x)}{\vol(G)} \cdot (\bE(\bv_{\textrm{in}} \rightarrow \bv_{\textrm{out}}) + \bE(\bv_{\textrm{out}} \rightarrow x) - \bE(\bv_{\textrm{in}} \rightarrow x)),
$$
where $(x \rightarrow y)$ is the time that a walk begun at $x$ hits $y$ for the first time.  (We adopt the convention that $(x \rightarrow x) = 0$.)  Furthermore, we have (\cite{CY}, Theorem 8)
$$
\bE(x \rightarrow y) = \frac{\vol(G)}{\tilde{\rho}(y)} G(y,y) - \frac{\vol(G)}{\tilde{\rho}(x)} G(x,y)
$$
where $G(x,y)$ is the {\it discrete Green's function} for $G$ with weights $\rho(\cdot)$, whence
\begin{align*}
\bE(\tau(x)) &= \tilde{\rho}(x) \cdot \left (
\frac{G(\bv_{\textrm{out}},\bv_{\textrm{out}})}{\tilde{\rho}(\bv_{\textrm{out}})} - \frac{G(\bv_{\textrm{in}},\bv_{\textrm{out}})}{\tilde{\rho}(\bv_{\textrm{in}})}
- \frac{G(\bv_{\textrm{out}},x)}{\tilde{\rho}(\bv_{\textrm{out}})}
+ \frac{G(\bv_{\textrm{in}},x)}{\tilde{\rho}(\bv_{\textrm{in}})}
\right ).
\end{align*}
The matrix $\bG$ of values $G(x,y)$ is given by (\cite{CY}, (16))
\begin{equation} \label{eq2}
\bG = T^{1/2} \cG T^{-1/2} = T^{1/2} \sum_{i = 1}^{n-1} \left ( \lambda_i^{-1} \phi_i^\ast \phi_i \right ) T^{-1/2},
\end{equation}
where $T = \diag(\tilde{\rho}_1,\ldots,\tilde{\rho}_{n})$, $0 = \lambda_0 < \lambda_1 \leq \cdots \leq \lambda_{n-1}$ are the eigenvalues of the normalized Laplacian $\cL$ and $\phi_0,\ldots,\phi_{n-1}$ are the corresponding eigenvectors.  The normalized Laplacian is, in turn, defined to be $T^{-1/2} L T^{-1/2}$ where $L$ is the combinatorial Laplacian:
$$
L(x,y) = \left \{ \begin{array}{ll} \tilde{\rho}(x) & \textrm{ if } x=y \\
                                    - \rho(x) \rho(y) & \textrm{ if } x \sim y \\
                                    0 & \textrm{ otherwise.} \end{array} \right .
$$

\section{Numerical Solution}

Ultimately, our objective is to reconstruct the function $\rho$ from $\hat{\tau}$.  There are $n-1$ unknowns that define $\rho(\cdot)$ (recall that $\rho(\bv_{\textrm{out}}) = 1$) and $n-1$ degrees of freedom in $\hat{\tau}$, so that such a reconstruction is reasonable to attempt.  A maximum-likelihood estimator for $\rho$ seems out of reach, however, since the relationship defining $\tau$ from $\rho$ is so complicated.  Therefore, we adopt a standard simplification: the method of moments.  That is, we try to solve $\bE(\tau) = \hat{\tau}$ for $\rho$.

This problem, though simpler, is still analytically intractable.  Nonetheless, one can approximate $\rho$ by iterative numerical methods.  Consider the following algorithm:
\begin{enumerate}
\item  For each solver, track how many times they visit each site $v \in V$ as they traverse the graph from $\bv_{\textrm{in}}$ to $\bv_{\textrm{out}}$.  Let the average number of visits for each group be $\hat{\tau}(v)$.
\item  Without loss of generality, we restrict our attention to the induced subgraph $G[\supp(f)]$.
\item  Start with a uniform distribution $\rho_0 : V \rightarrow [0,1]$, i.e., $\tau(v) \equiv 1$.
\item  Apply a steepest-descent strategy to the cost function $\vartheta$ (defined below).
\end{enumerate}
Meaningful information could be extracted from the resulting $\rho_{\textrm{final}}$ by, for example, performing a regression against some notion of distance to $\bv_{\textrm{out}}$: graphical distance, electrical resistance, etc.

Define $\tau_\rho = \bE(\tau)$, and let
\begin{equation} \label{eq1}
\vartheta(\rho) = \| \hat{\tau} - \tau_\rho \|^2_2 = (\hat{\tau} - \tau_\rho)^\ast (\hat{\tau} - \tau_\rho)
\end{equation}
where we are treating functions of $V$ as vectors in $\R^n$.  Let
$$
\Delta_x = \frac{d}{d\rho(x)} \vartheta(\rho).
$$
We may then apply, for example, steepest-descent, guided by the gradient vector $\ang{\Delta_x}_{x \in V}$.

Applying (\ref{eq1}), we have
\begin{align*}
\Delta_x &= (\hat{\tau} - \tau_\rho)^\ast \frac{d}{d\rho(x)} (\hat{\tau} - \tau_\rho) + \frac{d}{d\rho(x)} (\hat{\tau} - \tau_\rho)^\ast \cdot (\hat{\tau} - \tau_\rho) \\
& = (\tau_\rho - \hat{\tau})^\ast \frac{d\tau_\rho}{d\rho(x)} + \frac{d\tau_\rho^\ast}{d\rho(x)} (\tau_\rho - \hat{\tau}) = 2 (\tau_\rho - \hat{\tau})^\ast \frac{d\tau_\rho}{d\rho(x)}.
\end{align*}
To simplify this, first note that $d\tilde{\rho}(y)/d\rho(x)$ is $\rho(y)$ if $y \sim x$, $\sum_{y \sim x} \rho(y)$ if $y = x$, and $0$ otherwise.  Furthermore, $d\vol(G)/d\rho(x) = 2 \sum_{y \sim x} \rho(y)$.  We can then write
\begin{align*}
\frac{d\tau_\rho}{d\rho(x)} &= \frac{d}{d\rho(x)} \left (
\frac{\tilde{\rho}(x) G(\bv_{\textrm{out}},\bv_{\textrm{out}})}{\tilde{\rho}(\bv_{\textrm{out}})} - \frac{\tilde{\rho}(x) G(\bv_{\textrm{in}},\bv_{\textrm{out}})}{\tilde{\rho}(\bv_{\textrm{in}})} \right . \\
& \left . \qquad \qquad - \frac{\tilde{\rho}(x) G(\bv_{\textrm{out}},x)}{\tilde{\rho}(\bv_{\textrm{out}})}
+ \frac{\tilde{\rho}(x) G(\bv_{\textrm{in}},x)}{\tilde{\rho}(\bv_{\textrm{in}})}
\right ).
\end{align*}
Furthermore,
\begin{align*}
\frac{d}{d\rho(x)} \frac{\tilde{\rho}(x) G(a,b)}{\tilde{\rho}(a)} &= \left (\frac{d}{d\rho(x)} \frac{\tilde{\rho}(x)}{\tilde{\rho}(a)} \right ) \cdot G(a,b) + \frac{\tilde{\rho}(x)}{\tilde{\rho}(a)} \cdot \frac{d}{d\rho(x)} G(a,b). \\
&= \left (\tilde{\rho}(a) \frac{d}{d\rho(x)} \tilde{\rho}(x) - \tilde{\rho}(x) \frac{d}{d\rho(x)} \tilde{\rho}(a) \right ) \tilde{\rho}(a)^{-2} G(a,b) \\
& \qquad + \frac{\tilde{\rho}(x)}{\tilde{\rho}(a)} \cdot \frac{d}{d\rho(x)} G(a,b).
\end{align*}
Hence, it remains to compute $d\bG/d\rho(x)$.  To that end, we have the following result.  Define the {\it Moore-Penrose pseudoinverse} (or just {\it pseudoinverse} for short) of a real symmetric matrix $B$ of rank $n-k$ to be a matrix $A$ so that if $\bx_1,\ldots,\bx_k$ are an orthonormal basis for $\nl(B)$, then $AB = I - \sum_{j=0}^{k-1} \bx_j \bx_j^\ast$, and $\nl(A) = \nl(B)$.

\begin{theorem} \label{diffpseudo} Suppose that $A$ is the pseudoinverse of the real symmetric matrix $B$ with $\rank(B)=n-k$, then
$$
A^\prime = -(P + A B^\prime) A - A P^\prime,
$$
where $P = \sum_{j=0}^{k-1} \bx_j \bx_j^\ast$.
\end{theorem}
\begin{proof} Let $P_j = \bx_j \bx_j^\ast$, let $\bx_0,\ldots,\bx_{n-1}$ be orthonormal eigenvectors for $B$ (including the null vectors), and let $0 = \lambda_0 = \cdots = \lambda_{k-1} < \lambda_k < \cdots \lambda_{n-1}$ be the corresponding eigenvalues.  Then, differentiating $AB = I - \sum_j P_j$,
$$
(AB)^\prime = A^\prime B + A B^\prime = - P^\prime,
$$
whence $A^\prime B = -P^\prime - A B^\prime$.  Since $A = \sum_{j=k}^{n-1} \lambda_j^{-1} P_j$, $A$ is symmetric.  Hence,
\begin{align*}
BA &= \left ( \sum_{j=k}^{n-1} \lambda_j P_j \right ) A  \\
&= \sum_{j=k}^{n-1} \lambda_j \bx_j (\bx_j^\ast A) \\
&= \sum_{j=k}^{n-1} \lambda_j \lambda_j^{-1} \bx_j \bx_j^\ast \\
&= \sum_{j=k}^{n-1} \bx_j \bx_j^\ast = I - \sum_{j=0}^{k-1} P_j.
\end{align*}
Then right-multiplying by $A$ the expression for $A^\prime B$ above,
\begin{equation} \label{eq3}
A^\prime B A = A^\prime (I - P) = -(P^\prime + A B^\prime) A,
\end{equation}
so we may rewrite this as
$$
A^\prime = -(P^\prime + A B^\prime) A + A^\prime P.
$$
On the other hand, $A P = \sum_{j=0}^{k-1} A \bx_j \bx_j^\ast = \textbf{0}$, so
$$
A^\prime P = -A P^\prime,
$$
which we may apply to (\ref{eq3}) to get
$$
A^\prime = -(P^\prime + A B^\prime) A - A P^\prime.
$$
\end{proof}

Now,
\begin{align*}
\frac{d\bG}{d\rho(x)} &= \frac{1}{2} T^{-1/2} \frac{dT}{d\rho(x)} \cG T^{-1/2} + T^{1/2} \frac{d\cG}{d\rho(x)} T^{-1/2} - \frac{1}{2} T^{1/2} \frac{dT}{d\rho(x)} \cG T^{-3/2} \frac{dT}{d\rho(x)} \\
&= \frac{1}{2} \frac{dT}{d\rho(x)} T^{-1} \bG + T^{1/2} \frac{d\cG}{d\rho(x)} T^{-1/2} - \frac{1}{2} \bG T^{-1} \frac{dT}{d\rho(x)},
\end{align*}
since $T^\alpha$ and $T^\prime$ are diagonal, and therefore commute with each other.  The diagonal of $dT/d\rho(x)$ has $y$-coordinate $\rho(y)$ if $x \neq y$ and $x \sim y$, $\sum_{y \sim x} \rho(y)$ if $x = y$, and $0$ otherwise.  We may apply Theorem \ref{diffpseudo} to $\cG$, since $\cG$ is the pseudoinverse of $\cL$.  Then (abbreviating by the operator $d/d\rho(x)$ by $(\cdot)^\prime$),
$$
\cG^\prime = -(P^\prime + \cG \cL^\prime) \cG - \cG P^\prime
$$
where $P = \phi_0 \phi_0^\ast$.  In this expression,
$$
P^\prime = \phi_0^\prime \phi_0^\ast + \phi_0 \phi_0^{\prime \ast},
$$
and the $y$ coordinate of $\phi_0$ is $\sqrt{\tilde{\rho}(y)/\vol(G)}$, whence the $y$ coordinate of $\phi_0^\prime$ is
$$
\frac{d}{d \rho(x)} \sqrt{\frac{\tilde{\rho}(y)}{\vol(G)}} = \vol(G)^{-2} \left (\vol(G) \frac{d}{d \rho(x)} \tilde{\rho}(y) - \tilde{\rho}(y) \frac{d}{d \rho(x)} \vol(G) \right ).
$$
Finally, $\cL^\prime$ has $(y,z)$ entry $0$ if $y = z$, and, if $y \neq z$,
\begin{align*}
\frac{d}{d \rho(x)} \frac{\rho(y) \rho(z)}{\sqrt{\tilde{\rho}(y)\tilde{\rho}(z)}} &= (\tilde{\rho}(y)\tilde{\rho}(z))^{-1} \left ( \sqrt{\tilde{\rho}(y)\tilde{\rho}(z)} \frac{d}{d \rho(x)} \rho(y) \rho(z) \right . \\
& \quad \left . - \rho(y) \rho(z) \frac{d}{d \rho(x)} \sqrt{\tilde{\rho}(y)\tilde{\rho}(z)} \right ) \\
&= (\tilde{\rho}(y)\tilde{\rho}(z))^{-1} \left ( \sqrt{\tilde{\rho}(y)\tilde{\rho}(z)} \chi(y=x) \rho(z) \right . \\
& \quad + \sqrt{\tilde{\rho}(y)\tilde{\rho}(z)} \chi(z=x) \rho(y) \\
& \quad \left . - \frac{\rho(y) \rho(z) }{2 \sqrt{\tilde{\rho}(y)\tilde{\rho}(z)}} \left ( \tilde{\rho}(y) \frac{d \tilde{\rho}(z)}{d \rho(x)} + \tilde{\rho}(z) \frac{d \tilde{\rho}(y)}{d \rho(x)} \right ) \right ),
\end{align*}
where we are denoting the indicator function of an event $\cE$ by $\chi(\cE)$.\\

\section{Solution Existence}

It would be useful to know for certain that, for each $\hat{\tau}$, there does indeed exist a set of weights $\rho : V \rightarrow \R^{> 0}$ which gives rise to the desired expected visitation times.  In other words, we wish to show the existence of a $\rho$ so that
$$
E(\tau_\rho(x)) = \hat{\tau}(x).
$$
One could view such a $\rho$ as a Method-of-Moments estimator for the weight function of $V$.  Note that it is certainly impossible to solve for $\rho$ if $\hat{\tau}$ has disconnected support as an induced subgraph of $V$.  Indeed, the set of vertices visited by a random walk $\omega = (\bv_{\textrm{in}}=v_0,v_1,\ldots,v_{T-1},v_T = \bv_\textrm{out})$ is connected.  Write $\tr_\omega$, the ``trace'' of $\omega$ to be the function $\tr_\omega(\cdot) : V \rightarrow \Z$ whose value at $v$ is simply the number of occurrences of $v$ in $\omega$, i.e.,
$$
\tr_\omega(x) = |\{j : 0 \leq j \leq T, \, v_j = x\}|.
$$
We say that a walk $\omega = (\bv_{\textrm{in}},v_1,\ldots,v_{T-1},\bv_{\textrm{out}})$ is ``proper'' if $\tr_\omega(\bv_{\textrm{out}}) = 1$.  Recall that, for a function $\rho : V \rightarrow \R$ and $v \in V = V(G)$, we define $\tau_\rho(v)$ to be the expected number of visits to $v$ of a random walk that starts at $\bv_{\textrm{in}}$, navigates $G$ according to $\rho$, and ends at its first encounter with $\bv_{\textrm{out}}$.  We say that the equation $\tau_\rho = \br$ is ``solvable'' if there exists a $\rho$ with all positive coordinates so that the equation holds.  Note that we may restrict our attention to those $G$ so that $G^\prime = G \setminus \bv_{\textrm{out}}$ is connected, since any component of $G^\prime$ not containing $\bv_{\textrm{in}}$ cannot be visited by any proper walk.  Finally, define $\chi_v$ to be the characteristic function of the vertex $v \in V(G)$ and $\chi_e$ to be $\chi_x + \chi_y$ for any edge $e = \{x,y\} \in E(G)$.

For any choice of $\rho$, one can write
$$
\tau_\rho = \sum_{\textrm{proper } \omega} \tr_\omega \bP(\omega)
$$
where $\bP(\omega)$ is the probability that the walk $\omega$ occurs given the weighting $\rho$.  Therefore, if $\tau_\rho = \br$ is solvable, then $\br$ lies in the convex hull of the traces of all proper walks $\tr_\omega$.  It is not hard to see that $\br$ actually lies in $\Psi_G$, the interior {\it with respect to a minimal containing hyperplane} of the convex hull of the vectors $\tr_\omega \in \R^n$.  This minimal containing hyperplane $\cH$ is not full-dimensional, as the next lemma describes.

\begin{lemma} $\dim(\cH) = n-1$ if $G$ is not bipartite and $n-2$ if $G$ is bipartite.
\end{lemma}
\begin{proof} First of all, $\dim(\cH) \leq n-1$, since $\tr_\omega(\bv_{\textrm{out}}) = 1$.  We may write
$$
\cH = \tr_{\omega} + \textrm{span}(\{\tr_{\omega^\prime}-\tr_{\omega}\}_{\textrm{proper }{\omega^\prime}})
$$
for any proper walk $\omega$.  It therefore suffices to determine the dimension of $\cS = \textrm{span}(\{\tr_{\omega}-\tr_{\omega^\prime}\}_{\omega^\prime})$, for $\omega = (v_0,\ldots,v_T)$ some fixed proper walk which passes through every edge not incident to $\bv_{\textrm{out}}$.  Such a walk exists, since $G^\prime = G \setminus \bv_{\textrm{out}}$ is connected. Given an edge $e \in E(G^\prime)$, we define the walk $\omega_e$ by
$$
\omega_e = (v_0,v_1,\ldots,v_t,v_{t+1},v_t,v_{t+1},v_{t+2},\ldots,v_{T-1},v_T),
$$
where $t$ is the least index so that $\{v_t,v_{t+1}\} = e$.  Then
$$
\tr_{\omega_e} - \tr_{\omega} = \chi_{v_t} + \chi_{v_{t+1}} = \chi_e.
$$
Therefore, if $v$ is adjacent to $\bv_{\textrm{in}}$ in $G^\prime$, then $\chi_{\bv_{\textrm{in}}} + \chi_{v} \in \cS$.  If $v$ is adjacent to a vertex $w$ which is adjacent to $\bv_{\textrm{in}}$, then
$$
(\chi_{\bv_{\textrm{in}}} + \chi_v) - (\chi_v + \chi_w) = \chi_{\bv_{\textrm{in}}} - \chi_w \in \cS.
$$
Proceeding inductively, we see that, if there is a path of length $\ell$ from $\bv_{\textrm{in}}$ to $v$ in $G^\prime$, then
\begin{equation}
\chi_{\bv_{\textrm{in}}} - (-1)^{\ell} \chi_v \in \cS. \label{eq:vectorsinS}
\end{equation}
Since the functions $\chi_v$ are linearly independent for $v \in G^\prime$, this shows immediately that $\dim(\cS) \geq n-2$.

Suppose that $G$ is not bipartite.  Since $G^\prime$ is connected, there are two proper paths of length $\ell_1$ and $\ell_2$, where $\ell_1$ and $\ell_2$ differ in parity, from $\bv_{\textrm{in}}$ to $\bv_{\textrm{out}}$.  Therefore,
$$
\frac{1}{2} [(\chi_{\bv_{\textrm{in}}} - (-1)^{\ell_1} \chi_{\bv_{\textrm{out}}}) + (\chi_{\bv_{\textrm{in}}} - (-1)^{\ell_2} \chi_{\bv_{\textrm{out}}} )] = \chi_{\bv_{\textrm{in}}} \in S.
$$
Subtracting this quantity from (\ref{eq:vectorsinS}), we have that $\chi_{v} \in \cS$ for all $v \in V(G^\prime)$, so $\dim(\cH) = n-1$.  On the other hand, if $G$ is bipartite, then there is a function $c: V(G) \rightarrow \{-1,1\}$ inducing the bipartition.  For any proper walk $\omega^\prime = (w_0,\ldots,w_{T^\prime})$, $c(w_j)$ alternates as $j$ goes from $0$ to $t$.  Hence, $c \cdot \tr_{\omega^\prime} \in \{-1,0,1\}$ (where we think of both factors in this dot product as vectors in $\R^n$) has the same value for any proper walk $\omega^\prime$.  We may conclude that
$$
c \cdot (\tr_{\omega} - \tr_{\omega^\prime}) = 0,
$$
so that $\cS \perp \textrm{span}\{c,\chi_{\bv_{\textrm{out}}}\}$.  Since this span is clearly two-dimensional, $\dim(\cH) = \dim(\cS) = n-2$.
\end{proof}

\begin{conj} \label{mainconjecture} The equation
$$
\tau_\rho = \br
$$
is solvable if and only if $\br$ lies in the relative interior of the convex hull of the $\tr_\omega \in \R^n$, for all proper walks $\omega$.
\end{conj}

Necessity is immediate, by considering the set of all proper walks weighted by their probabilities.  We begin our attack on sufficiency modestly.  Define
$$
\rho^\ast(w) := \frac{\tilde{\rho}(w)}{\rho(w)} = \sum_{z \sim w} \rho(z),
$$
and write $\be_w$ for the elementary vector with nonzero coordinate at $w \in V$, i.e., the indicator function of $w$.

\begin{theorem} Fix $\br \in \R^{V(G)}$.  Let $\alpha \geq 0$ and suppose that $\tau_\rho = \br - \alpha \be_v$ is solvable.  Let $G^\prime$ be the graph obtained from $G$ by attaching a vertex $v^\prime$ of degree $1$ to $v \in V$, and let $\br^\prime : G^\prime \rightarrow \R$ be defined by
$$
\br^\prime(w) = \left \{ \begin{array}{ll} \br(w) & \textrm{if } w \in G \\ \alpha & \textrm{if } w = v^\prime \end{array} \right . .
$$
Then $\tau_{\rho^\prime} = \br^\prime$ is solvable (whence $\br^\prime \in \Psi(G^\prime)$).
\end{theorem}
\begin{proof} By hypothesis, we can solve
$$
\tau_\rho(w) = \br - \alpha \be_v = \left \{ \begin{array}{ll} \br(w) & \textrm{if } w \in G \setminus v \\ \br(v) - \alpha & \textrm{if } w = v \end{array} \right .
$$
for $\rho$.  Define $\rho^\prime |_G = \rho |_G$ and
$$
\rho^\prime(v^\prime) = \frac{\rho^\ast(v) \alpha}{\br(v) - \alpha}.
$$
Call a visit to $v$ ``initial'' if it is not immediately preceded by a visit to $v^\prime$.  Note that, at every visit to $v$ of a random walk according to $\rho^\prime$, the probability of visiting $v^\prime$ on the next step is  $\frac{\rho^\prime(v^\prime)}{\rho^\prime(v^\prime) + \rho^\ast(v)}$.  Hence, the expected number of visits to $v^\prime$ that occur with each initial visit to $v$ is
\begin{align*}
\sigma := \sum_{k \geq 1} \left ( \frac{\rho^\prime(v^\prime)}{\rho^\prime(v^\prime) + \rho^\ast(v)} \right )^k &= \frac{\rho^\prime(v^\prime)}{\rho^\prime(v^\prime) + \rho^\ast(v)} \left ( \frac{1}{1 - \frac{\rho^\prime(v^\prime)}{\rho^\prime(v^\prime) + \rho^\ast(v)}} \right ) \\
& = \frac{\rho^\prime(v^\prime)}{\rho^\prime(v^\prime) + \rho^\ast(v)} \cdot \frac{\rho^\prime(v^\prime) + \rho^\ast(v)}{\rho^\ast(v)} \\
& = \frac{\rho^\ast(v) \alpha / (\br(v) - \alpha)}{\rho^\ast(v)} = \frac{\alpha}{\br(v) - \alpha}.
\end{align*}
Now, if we excise from the walks according to $\rho^\prime$ the steps immediately following each initial visit to $v$ up until (but not including) the next time that the walk is neither at $v$ nor $v^\prime$, the distribution of the resulting walks proceeds according to $\rho$ on $G$.  It is easy to see then that there are an expected $\br(v)- \alpha$ number of initial visits to $v$ in a walk according $\rho^\prime$, which implies that
$$
\tau_{\rho^\prime}(v^\prime) = (\br(v) - \alpha) \cdot \sigma = \alpha.
$$
On the other hand, since each visit to $v^\prime$ is immediately followed by a visit to $v$, the expected number of visits to $v$ under $\rho^\prime$ is simply
$$
\tau_{\rho^\prime}(v) = \tau_{\rho}(v) (1 + \sigma) = \left ( \br(v) - \alpha \right ) \left ( \frac{\br(v)}{\br(v) - \alpha} \right ) = \br(v).
$$
Finally, since projecting the $\rho^\prime$-walk onto $G$ via excision (as described above) yields a $\rho$-walk,
$$
\tau_{\rho^\prime}(w) = \tau_{\rho}(w) = \br(w)
$$
for each $w \in G \setminus \{ v,v^\prime \}$.  This in turn implies that $\tau_{\rho^\prime} = \br^\prime$ is solvable.
\end{proof}

\begin{cor} Suppose that $\tau_{\rho} = \br$ is solvable for every $\br \in \Psi(G)$.  Let $G^\prime$ be the graph obtained from $G$ by attaching a vertex $v^\prime$ of degree $1$ to $v \in V$, and let $\br^\prime : G^\prime \rightarrow \R$ be defined by
$$
\br^\prime(w) = \left \{ \begin{array}{ll} \br(w) & \textrm{if } w \in G \\ \alpha & \textrm{if } w = v^\prime \end{array} \right . .
$$
If $\br^\prime \in \Psi(G^\prime)$, then $\tau_{\rho^\prime} = \br^\prime$ is solvable.
\end{cor}
\begin{proof} By the preceding theorem, we need only show that $\br^\prime \in \Psi(G^\prime)$ implies $\br - \alpha \be_v \in \Psi(G)$.  Therefore, suppose that $\br^\prime \in \Psi(G^\prime)$, so we may write
$$
\br^\prime = \sum_\omega \lambda_\omega \tr_\omega
$$
where
$$
\sum_\omega \lambda_\omega = 1.
$$
For each $G^\prime$-walk $\omega$, let $\widetilde{\omega}$ be the walk obtained from $\omega$ by the excision process described in the preceding proof.  Note that
$$
\tr_{\widetilde{\omega}} = \tr_{\omega}|_{V(G)} - \tr_{\omega}(v^\prime) \be_v.
$$
Hence,
\begin{align*}
\sum_\omega \lambda_\omega \tr_{\widetilde{\omega}} &= \sum_\omega \lambda_\omega \left ( \tr_{\omega}|_{V(G)} - \tr_{\omega}(v^\prime) \be_v \right ) \\
&= \sum_\omega \lambda_\omega \tr_{\omega}|_{V(G)} - \sum_\omega \lambda_\omega \tr_{\omega}(v^\prime) \be_v \\
&= \br^\prime |_{V(G)} - \br^\prime(v^\prime) \be_v \\
&= \br - \alpha \be_v.
\end{align*}
To see that the point $\br - \alpha \be_v$ is actually in the {\it interior} of the convex hull, simply note that the open mapping theorem implies that the map $(x_1,\ldots,x_n) \mapsto (x_1,\ldots,x_{n-2},x_{n-1} - x_1)$ (and any map obtained by permuting coordinates) from the minimal containing hyperplane of $\Psi_G$ to its image preserves open sets.  The conclusion follows immediately.
\end{proof}

\begin{theorem} Assume that $G$ has two vertices $v, w \in G \setminus \{a,b\}$ such that $N(v) = N(w)$.  Further suppose that $\tau_{\rho} = \br^\prime$ is solvable for every $\br^\prime \in \Psi(G^\prime)$, where $G^\prime = G - w$.  If $\br \in \Psi_G$, then $\tau_\rho = \br$ is solvable.
\end{theorem}
\begin{proof}
Define $r^\prime : V(G^\prime) \rightarrow \R$ to be
$$
r^\prime(x) = \left \{ \begin{array}{ll} \br(x) & \textrm{if } x \neq v \\ \br(v)+\br(w) & \textrm{if } x = v . \end{array} \right .
$$
Since $\br^\prime \in \Psi(G^\prime)$, we can write $\br^\prime = \sum_\omega \lambda_\omega \tr_\omega$.  It is easy to see that, if we write $\omega^\prime$ for the walk obtained from $\omega$ by replacing each occurrence of $w$ with $v$, then
$$
\br = \sum_\omega \lambda_\omega \tr_{\omega^\prime}.
$$
Since $v$ and $w$ have identical neighborhoods, $\omega^\prime$ is a bona fide $G$-walk for each $G^\prime$-walk $\omega$.  The open mapping theorem implies that the map $(x_1,\ldots,x_n) \mapsto (x_1,\ldots,x_{n-2},x_{n-1} + x_1)$ (and any map obtained by permuting coordinates) from the minimal containing hyperplane of $\Psi_G$ to its image preserves open sets.  Hence, $\br \in \Psi(G)$, and, by hypothesis, we can solve $\tau_{\rho} = \br^\prime$.

Now, for $\alpha \in [0,1]$, let $\rho_\alpha$ agree with $\rho$ on $G \setminus \{v,w\}$, $\rho_\alpha(v) = \alpha \rho(v)$, and $\rho_\alpha(w) = 1 - \alpha \rho(v)$.  A $\rho_\alpha$-walk visits the set $\{v,w\}$ an expected $\br(v) + \br(w)$ number of times, with each visit going to $v$ with probability $\alpha$ and going to $w$ with probability $1 - \alpha$.  Therefore, the expected number of visits to $v$ is $(\br(v) + \br(w)) \alpha$ and the expected number of visits to $w$ is $(\br(v) + \br(w)) (1- \alpha)$.   We can set $\alpha = \br(v)/(\br(v) + \br(w))$ so that $\tau_{\rho_\alpha} = \br$.
\end{proof}

Dealing with the case of a path would be useful at this point.  In that case, we write the vertices of $G$ in order: $v_1 = \bv_{\textrm{out}}, v_2, \ldots, v_{n-1}, v_n = \bv_{\textrm{in}}$.  Write $\R^+$ for the nonnegative reals and $\R^{++}$ for the positive reals.

\begin{prop} \label{prop:cone} Suppose $G$ is a path.  The set $\Psi_G$ is precisely the set of vectors of the form $\1 + \sum_{j=2}^{n-1} \alpha_j \Bf_j$, where $\Bf_j = \be_j + \be_{j+1}$ and $\alpha_j > 0$ for each $2 \leq j \leq n-1$.
\end{prop}
\begin{proof} We actually show that the topological closure $\bar{\Psi}$ of $\Psi_G$ is of the form $\1 + \sum_{j=2}^{n-1} \R^+ \Bf_j$.  It is easy to see that the conclusion $\1 + \sum_{j=2}^{n-1} \R^{++} \Bf_j = \Psi_G$ then follows, since non-boundary points $\bx$ can be perturbed by some $\sum_{j=2}^{n-1} \epsilon_j \Bf_j$ for $\epsilon_j > 0$ without leaving the set, implying that the projection of $\bx - \1$ onto each $\Bf_j$ is nonzero.

Let $\eta(j,k)$ denote the walk from $\bv_{\textrm{in}}$ to $\bv_{\textrm{out}}$ of the form
$$
(v_n, v_{n-1}, \ldots, v_{j+2}, \underbrace{v_{j+1}, v_j, \ldots, v_{j+1}, v_j}_{k \, \textrm{times}}, v_{j-1}, \ldots, v_2, v_1),
$$
that is, a direct path with $k$ ``steps back'' at $j$ added, $k \geq 2$ and $2 \leq j \leq n-1$.  (Write $\eta$ for the path with no steps backwards.)  Clearly, $\tr_{\eta(j,k)} = k \Bf_j + \1$.  By taking convex combinations of $\tr_{\eta(j,k)}$ and $\tr_\eta$ for sufficiently large $k$, one can construct any $\alpha \Bf_j + \1$ with $\alpha \geq 0$.  Then, by taking convex combinations of the resulting vectors, the inclusion $\1 + \sum_{j=2}^{n-1} \R^+ \Bf_j \subset \bar{\Psi}$ follows.

For the opposite inclusion, it suffices to show that $\tr_\omega \in \1 + \sum_{j=2}^{n-1} \R^+ \Bf_j$ for each proper walk $\omega$.  We show this inductively: if $\omega = \eta$, the statement evidently holds.  Hence, assume that, for some $t > 0$, $\omega(t+2) = \omega(t) = v_j$ and $\omega(t+1) = v_{j+1}$.  Every proper walk other than $\eta$ admits such a $t$ since, for example, we may take $v_j \rightarrow v_{j+1} \rightarrow v_j$ to be the last step backwards.  Then
$$
\tr_\omega = \tr_{\omega^\prime} + \Bf_j
$$
where $\omega^\prime$ is the proper walk $\omega$ with steps $t+1$ and $t+2$ removed.  Clearly, by iterating this argument, we arrive at a representation of the form
$$
\tr_\omega = \1 + \sum_{j=2}^{n-1} \alpha_j \Bf_{j}
$$
with $\alpha_j \geq 0$.
\end{proof}

We need the following lemma, which allows us to compute expected occupation time vectors as eigenvectors of a certain matrix.

\begin{lemma} \label{lemma:mreqr} Let $n = |V(G)|$, where $G$ is a weighted graph with $\wt(v) = \beta_v$ and distinguished vertices $\bv_{\textrm{in}}$, $\bv_{\textrm{out}}$.  There is a unique nonnegative vector $\br \in \R^{V(G)}$ so that $\br_{\bv_{\textrm{out}}} = 1$, $\|\br\|_1 > 1$, and $M \br = \br$, where $M \in \R^{n\times n}$ is defined by
$$
M_{vw} = \left \{ \begin{array}{l} 1 \textrm{ if } (v,w)=(\bv_{\textrm{out}},\bv_{\textrm{out}}) \textrm{ or } (v,w) = (\bv_{\textrm{in}},\bv_{\textrm{out}}) \\
                                \frac{\beta_v}{\sum_{u \sim w} \beta_u} \textrm{ if } v \sim w \textrm{ and } w \neq \bv_{\textrm{out}}\\
                                0 \textrm{ otherwise.} \end{array} \right .
$$
Furthermore, $\br_v$ is the expected number of visits in a proper random walk on $G$ with weights $\{\beta_v\}_{v \in V(G)}$.
\end{lemma}
\begin{proof} Let $\Gamma$ be the weighted digraph whose adjacency matrix is $M$.  Then $\Gamma$ consists of $G$ with each edge incident to $\bv_{\textrm{out}}$ removed, plus a single directed edge from ${\bv_{\textrm{out}}}$ to ${\bv_{\textrm{in}}}$.  In particular, $M^n$ has all positive entries, except for the nondiagonal elements of its first column (which are $0$).  Suppose $\br_{\bv_{\textrm{out}}} = \br^\prime_{\bv_{\textrm{out}}} = 1$, $\|\br\| > 1$, $\|\br^\prime\| > 1$, $M \br = \br$, $M \br^\prime = \br^\prime$, but $\br \neq \br^\prime$.  Then $M^n (\br-\br^\prime) = \br - \br^\prime$ as well.  Note that there exists some strictly positive vector $\bu$ so that $\bu \cdot \br = \bu \cdot \br^\prime = 1$: simply choose a positive vector orthogonal to $\br - \br^\prime$ and scale it so that its dot product with $\br$ is $1$.  (The vector $\br - \br^\prime$ has positive and negative entries since $\br$ and $\br^\prime$ each have at least two positive entries, and they are not the same vector.)  If we replace the first row of $M^n$ with the vector $\bu$, obtaining a new matrix $M^\prime$, then $M^\prime (\br - \br^\prime) = \br - \br^\prime$.  The Perron-Frobenius Theorem implies that $\br$ has all positive entries.  However, its first coordinate is $0$, a contradiction unless $\br = \br^\prime$, which is also a contradiction.  Hence, the solution to  $M \br = \br$ is unique.

We therefore need only show that the vector $\br$ of expected number of visits satisfies $M\br = \br$, since $\br_{\bv_{\textrm{out}}} = 1$ and $\br$ has additional nonzero entries.  It is clear that $(M \br)_{\bv_{\textrm{out}}} = 1 = \br_{\bv_{\textrm{out}}}$.  Since $\br_w$ is the $\beta_v/\sum_{u \sim w} \beta_u$-weighted sum of $\br_v$ for $v \sim w$, where $v \neq \bv_{\textrm{in}}, \bv_{\textrm{out}}$, the claim also holds for these $v$'s.  As for $\br_{\bv_{\textrm{in}}}$, the expected number of visits to $\bv_{\textrm{in}}$ is the weighted sum of its neighbors' expected number of visits, plus $1$, since the first visit to $\br_{\bv_{\textrm{in}}}$ is not preceded by a visit to any other vertex.  However, this extra ``$1$'' comes from the $(v,w) = (\bv_{\textrm{in}},\bv_{\textrm{out}})$ term in $M$, because $\br_{\bv_{\textrm{out}}}=1$.
\end{proof}

\begin{theorem} For $G$ a path, $\tau_\rho = \br$ is solvable iff $\br \in \Psi_G$.
\end{theorem}
\begin{proof} By Proposition \ref{prop:cone}, we may assume that $\br = \1 + \sum_{j=2}^{n-1} \alpha_j \Bf_j$ with $\alpha_j > 0$ for all $j$, $2 \leq j \leq n-1$.  Let $\rho(v_1) = 1$, $\rho(v_2) = 1$, and, for $j > 2$,
$$
\rho(v_j) = \frac{\prod_{k=0}^{\floor{(j-3)/2}} \alpha_{j-2k-1}}{\prod_{k=0}^{\floor{(j-4)/2}} (1+\alpha_{j-2k-2})}.
$$
where we interpret an empty product as $1$.  Let $\beta_j = \rho(v_j)$.  To see that $\tau_\rho = \br$, we need to show that $M \br = \br$, where $M$ is the matrix given by
$$
\left [ \begin{array}{cccccccccc}
1 & 0 & 0 & 0 & 0 & \cdots & 0 & 0 & 0 & 0 \\
0 & 0 & \frac{\beta_2}{\beta_2 + \beta_4} & 0 & 0 & \cdots & 0 & 0 & 0 & 0 \\
0 & \frac{\beta_3}{\beta_1 + \beta_3} & 0 & \frac{\beta_3}{\beta_3 + \beta_5} & 0 & \cdots & 0 & 0 & 0 & 0 \\
0 & 0 & \frac{\beta_4}{\beta_2 + \beta_4} & 0 & \frac{\beta_4}{\beta_4 + \beta_6} & \cdots & 0 & 0 & 0 & 0 \\
\vdots & \vdots & \vdots & \vdots & \vdots & \ddots & \vdots & \vdots & \vdots & \vdots \\
0 & 0 & 0 & 0 & 0 & \cdots & \frac{\beta_{n-2}}{\beta_{n-4}+\beta_{n-2}} & 0 & \frac{\beta_{n-2}}{\beta_{n-2} + \beta_n} & 0 \\
0 & 0 & 0 & 0 & 0 & \cdots & 0 & \frac{\beta_{n-1}}{\beta_{n-3}+\beta_{n-1}} & 0 & 1 \\
1 & 0 & 0 & 0 & 0 & \cdots & 0 & 0 & \frac{\beta_{n}}{\beta_{n-2}+\beta_{n}} & 0
\end{array} \right ]
$$
This will suffice to provide the result, since by Lemma \ref{lemma:mreqr}, $\br$ is the unique solution to $M \br = \br$ with $\br(b)=1$.

Recall that
$$
\br = \left [ \begin{array}{c} 1 \\ 1 + \alpha_2 \\ 1 + \alpha_2 + \alpha_3 \\ \vdots \\ 1 + \alpha_{n-2} + \alpha_{n-1} \\ 1 + \alpha_{n-1}
\end{array} \right ].
$$
Hence, the first coordinate of $M \br$ is $1 = \br(1)$.  The second coordinate of $M \br$ is
\begin{align*}
\frac{\beta_2}{\beta_2 + \beta_4} \br(3) &= \frac{1}{1 + \frac{\alpha_3}{1 + \alpha_2}} (1 + \alpha_2 + \alpha_3) \\
& = \frac{1 + \alpha_2}{1 + \alpha_2 + \alpha_3} (1 + \alpha_2 + \alpha_3) = 1 + \alpha_2 = \br(2).
\end{align*}
Let $p_j = \beta_j/(\beta_{j-2} + \beta_j)$ and $q_j = \beta_j/(\beta_j + \beta_{j+2}) = 1 - p_{j+2}$.  Then
\begin{align*}
p_j &= \frac{\beta_j}{\beta_{j-2} + \beta_j} = \left ( 1 + \beta_{j-2}/\beta_j \right )^{-1}\\
&= \left ( 1 + \frac{\prod_{k=0}^{\floor{(j-5)/2}} \alpha_{j-2k-3}}{\prod_{k=0}^{\floor{(j-6)/2}} (1+\alpha_{j-2k-4})} \cdot \frac{\prod_{k=0}^{\floor{(j-4)/2}} (1+\alpha_{j-2k-2})}{\prod_{k=0}^{\floor{(j-3)/2}} \alpha_{j-2k-1}} \right )^{-1} \\
&= \left ( 1 + \frac{1+\alpha_{j-2}}{\alpha_{j-1}} \right )^{-1} = \frac{\alpha_{j-1}}{1 + \alpha_{j-2} + \alpha_{j-1}},
\end{align*}
and
$$
q_j = 1 - p_{j+2} = \frac{1 + \alpha_{j}}{1 + \alpha_{j} + \alpha_{j+1}}.
$$
Then, for $3 \leq j \leq n-2$, the $j^\textrm{th}$ coordinate of $M \br$ is given by
\begin{align*}
p_j \br(j-1) + q_j \br(j+1) &= \frac{\alpha_{j-1}}{1 + \alpha_{j-2} + \alpha_{j-1}} (1 + \alpha_{j-2} + \alpha_{j-1}) \\
& \qquad + \frac{1 + \alpha_{j}}{1 + \alpha_{j} + \alpha_{j+1}} (1 + \alpha_j + \alpha_{j+1}) \\
& = 1 + \alpha_{j-1} + \alpha_j = \br(j).
\end{align*}
The $(n-1)^\textrm{st}$ coordinate of $M \br$ is
\begin{align*}
p_{n-1} \br(n-2) + \br(n) &= \frac{\alpha_{n-2}}{1 + \alpha_{n-3} + \alpha_{n-2}} (1 + \alpha_{n-3} + \alpha_{n-2}) + (1 + \alpha_{n-1}) \\
&= 1 + \alpha_{n-2} + \alpha_{n-1} = \br(n-1),
\end{align*}
and the $n^\textrm{th}$ coordinate of $M \br$ is
\begin{align*}
1 + p_{n} \br(n-1) &= 1 + \frac{\alpha_{n-1}}{1 + \alpha_{n-2} + \alpha_{n-1}} (1 + \alpha_{n-2} + \alpha_{n-1})\\
&= 1 + \alpha_{n-1} = \br(n).
\end{align*}

\end{proof}

\begin{theorem} For $G=K_n$ a complete graph with vertices $v_1=\bv_{\textrm{out}},v_2=\bv_{\textrm{in}},v_3,\ldots,v_n$, $\tau_\rho = \br$ is solvable iff $\br \in \Psi_G$.
\end{theorem}
\begin{proof}
First we give a description of the solvable $\br$'s.  In order to simplify our calculations, we will assume (without loss of generality) that the weights $\beta_1,\ldots,\beta_n$ sum to $1$.  Therefore, define $M$ to be
$$
\left [ \begin{array}{cccccc}
1 & 0 & 0 & \cdots & 0 & 0 \\
1 & 0 & \frac{\beta_2}{1-\beta_3} & \cdots & \frac{\beta_2}{1-\beta_{n-1}} & \frac{\beta_2}{1-\beta_n} \\
0 & \frac{\beta_3}{1-\beta_2} & 0 & \cdots & \frac{\beta_3}{1-\beta_{n-1}} & \frac{\beta_3}{1-\beta_n} \\
\vdots & \vdots & \vdots & \ddots & \vdots & \vdots \\
0 & \frac{\beta_{n-1}}{1-\beta_2} & \frac{\beta_{n-1}}{1-\beta_3} & \cdots & 0 & \frac{\beta_{n-1}}{1-\beta_n} \\
0 & \frac{\beta_n}{1-\beta_2} & \frac{\beta_n}{1-\beta_3} & \cdots & \frac{\beta_n}{1-\beta_{n-1}} & 0 \\
\end{array} \right ]
$$
so that $M \br = \br$ by Lemma \ref{lemma:mreqr}.  Note that the lower-right $(n-1)\times(n-1)$ submatrix of $M$ is simply
$$
\left [ \beta_2 \cdots \beta_n \right ] (J-I) \left [ \begin{array}{c} (1-\beta_2)^{-1} \\ \vdots \\ (1-\beta_n)^{-1} \end{array} \right ],
$$
where $J \in \R^{(n-1)\times(n-1)}$ is the all ones matrix, and $I$ is the identity.  We show that the following is a solution to $M \br = \br$ (and therefore the unique one with $\br_1 = 1$):
$$
\br_j = \left \{ \begin{array}{ll} \beta_j (1-\beta_j)/\beta_1 & \textrm{ if } j \neq 1,2 \\
                                   1                           & \textrm{ if } j = 1 \\
                                   (1 + \beta_2/\beta_1)(1-\beta_2) & \textrm{ if } j=2.
                 \end{array} \right .
$$
It is clear that $(M\br)_1 = 1 = \br_1$.  If $j \neq 1,2$, then
\begin{align*}
(M \br)_j &= \frac{\beta_j (1 + \beta_2/\beta_1)(1-\beta_2)}{1-\beta_2} + \sum_{\substack{i = 3\\i \neq j}}^n \frac{\beta_i(1-\beta_i)}{\beta_1} \cdot \frac{\beta_j}{1-\beta_i} \\
&= \frac{\beta_j (\beta_1 + \beta_2)}{\beta_1} + \sum_{\substack{i = 3\\i \neq j}}^n \frac{\beta_i \beta_j}{\beta_1} \\
&= \frac{\beta_j}{\beta_1} \left ( \beta_1 + \beta_2 + (1-\beta_1-\beta_2-\beta_j) \right ) = \frac{\beta_j (1-\beta_j)}{1-\beta_1} = \br_j.
\end{align*}
It remains to check $\br_2$:
\begin{align*}
(M \br)_2 &= 1 + \sum_{i = 3}^n \frac{\beta_i(1-\beta_i)}{\beta_1} \cdot \frac{\beta_2}{1-\beta_i} \\
&= 1 + \frac{(1-\beta_1-\beta_2) \beta_2}{\beta_1} = \frac{(\beta_1+\beta_2)(1-\beta_2)}{\beta_1} = \br_2.
\end{align*}
Now, if $\br_j = \beta_j (1-\beta_j)/\beta_1$ for each $j\geq 3$, then
$$
\beta_j = \frac{1 \pm \sqrt{1 - 4 \br_j \beta_1}}{2},
$$
with $\pm$ interpreted to be addition if $\beta_j > 1/2$ and subtraction otherwise.  Note that at most one of the $\beta_j$ can exceed $1/2$ since $\sum_{j=1}^n \beta_j = 1$.

Suppose for the moment that $\beta_j \leq 1/2$ for $j \geq 3$.  Then we can write $\beta_2 = 1 - \sum_{j \neq 2} \beta_j$, whence
\begin{align*}
\br_2 &= \frac{(\beta_1 + \beta_2)(1-\beta_2)}{\beta_1}\\
&= \frac{(1 - \sum_{j \neq 1,2} \beta_j)\sum_{j \neq 2} \beta_j}{\beta_1}\\
&= \frac{(2 - \sum_{j \neq 1,2} (1 - \sqrt{1 - 4 \br_j \beta_1}))(2 \beta_1 + \sum_{j \neq 1,2} (1 - \sqrt{1 - 4 \br_j\beta_1}))}{4 \beta_1}.
\end{align*}
This expression is defined for $\beta_1 \in (0,\min_{j \geq 3}{(4\br_j)^{-1}}]$.  We will assume for convenience that $\br_3 = \max_{j \geq 3} \br_j$, so $0 < \beta_1 \leq \br_3^{-1}/4$.  Let $u_j = 1 - \sqrt{1 - 4 \br_j \beta_1}$.  Then, when $\beta_1 \rightarrow 0^+$, we have
\begin{align*}
\lim_{\beta_1 \rightarrow 0^+} \br_2 &= \lim_{\beta_1 \rightarrow 0^+} \frac{(2 - \sum_{j \neq 1,2} u_j) (2 \beta_1 + \sum_{j \neq 1,2} u_j)}{4 \beta_1} \\
&= \frac{1}{4} \left [ -(2\beta_1 + \sum_{j \neq 1,2} u_j) \sum_{j \neq 1,2} \frac{du_j}{d\beta_1} + (2 - \sum_{j \neq 1,2} u_j) (2 + \sum_{j \neq 1,2} \frac{du_j}{d\beta_1})  \right ]_{\beta_1 = 0},
\end{align*}
where $u_j = 1 - \sqrt{1 - 4 \br_j \beta_1}$, by L'H\^{o}pitals' Rule.  Since $u_j|_{\beta_1=0} = 0$ and
$$
\left . \frac{du_j }{d\beta_1} \right |_{\beta_1 = 0} = \left . \frac{2\br_j}{\sqrt{1-4\br_j \beta_1}} \right |_{\beta_1 = 0}= 2\br_j,
$$
we have
$$
\lim_{\beta_1 \rightarrow 0^+} \br_2 = \frac{1}{4} (0 + 2 \cdot (2 + \sum_{j \neq 1,2} 2\br_j)) = 1 + \sum_{j \neq 1,2} \br_j.
$$
On the other hand, if $\beta_1 = (4\br_3)^{-1}$, then
\begin{equation} \label{equpperbeta1}
\br_2 = (2 - \sum_{j > 3} (1 - \sqrt{1 - \br_j / \br_3}) \left (\frac{1}{2} + \br_3 + \br_3 \sum_{j >3} (1 - \sqrt{1 - \br_j/\br_3} ) \right).
\end{equation}

Now, if $\beta_j > 1/2$ for some $j > 2$, we may assume without loss of generality that $j=3$.  Note that
$$
\br_j = \beta_j (1-\beta_j)/\beta_1 \geq \alpha(1-\alpha)/\beta_1
$$
for all $\alpha \leq 1- \beta_j$.  But $\beta_i \leq 1- \beta_j$ for all $i \neq j$, so $r_j \geq r_i$ for all $i \neq 1,2,j$.  Then, we have again that $\br_3 = \max_{j \geq 3} \br_j$.  Letting $u_j$ be as above for $j > 3$ and $u_3 = 1 + \sqrt{1 - 4 \br_3 \beta_1}$, we have
\begin{align*}
\br_2 &= \frac{(\beta_1 + \beta_2)(1-\beta_2)}{\beta_1}\\
&= \frac{(1 - \sum_{j \neq 1,2} \beta_j)\sum_{j \neq 2} \beta_j}{\beta_1}\\
&= \frac{(2 - \sum_{j \geq 3} u_j)(2 \beta_1 + \sum_{j \geq 3} u_j)}{4 \beta_1}.
\end{align*}
This expression is again defined for any $\beta_1 \in (0,{(4\br_3)^{-1}}]$.  The above expression agrees with (\ref{equpperbeta1}) when $\beta_1 = (4\br_3)^{-1}$, since then $u_3 = 0$.  On the other hand, when $\beta_1 \rightarrow 0^+$,
\begin{align*}
\lim_{\beta_1 \rightarrow 0^+} \br_2 &= \lim_{\beta_1 \rightarrow 0^+} \frac{(2 - \sum_{j \geq 3} u_j)(2 \beta_1 + \sum_{j \geq 3} u_j)}{4 \beta_1} \\
&= \frac{1}{4} \left [ -(2\beta_1 + \sum_{j \neq 1,2} u_j) \sum_{j \neq 1,2} \frac{du_j}{d\beta_1} + (2 - \sum_{j \neq 1,2} u_j) (2 + \sum_{j \neq 1,2} \frac{du_j}{d\beta_1})  \right ]_{\beta_1 = 0}\\
&= -\frac{1}{2} \left . \sum_{j \neq 1,2} \frac{du_j}{d\beta_1} \right |_{\beta_1 = 0},
\end{align*}
since $u_j = 0$ when $\beta_1 = 0$ except for $u_3$, which is $2$.  Now, $\frac{du_j}{d\beta_1} = 2r_j$ for $j > 3$, but $\frac{du_3}{d\beta_1} = -2r_3$.  Hence,
$$
\lim_{\beta_1 \rightarrow 0^+} \br_2 = \br_3 - \sum_{j > 3} \br_j.
$$
We may conclude, by the Intermediate Value Theorem, that $\br > 0$ is solvable as long as $\br_1 = 1$ and
$$
\br_3 - \sum_{j > 3} \br_j < \br_2 < 1 + \sum_{j > 2} \br_j.
$$
We claim that this inequality holds for all elements of $\Psi_G$.  To see the upper inequality, consider the fact that each visit (after the first) to $v_2$ of a proper walk is preceded by a visit to some $v_j$ with $j > 2$.  Hence $\br_2$ is at most one more than $\sum_{j > 2} \br_j$.  To see the lower inequality, we write it thusly:
$$
\br_3 \leq \sum_{j \neq 1,3} \br_j.
$$
Again, every visit to $v_3$ in a proper walk is preceded by a visit to some $v_j$ with $j \neq 1,3$.  The inequality, and the theorem, follows.
\end{proof}

\section{Open problems}

The following are unsolved problems that have arisen in the current study and which we would like to see addressed.

\begin{enumerate}
\item Conjecture \ref{mainconjecture}: For which $\br$ is it possible to solve for the weights in the equation $\tau_\rho = \br$?
\item Is it true that the iterated numerical solution described above always yields the correct answer, assuming a solution exists?  To put it another way, is there a unique local minimizer of $\|\tau_\rho - \br\|_2^2$ for a given $\br$?
\item If more information is available about the routes that random walkers take than just the empirical mean occupation times, could one exploit this to more efficiently obtain the weights, or to obtain a ``better'' set of weights?
\item Suppose some measure of expertise is used after the weights are obtained.  For example, one might ask for the correlation coefficient between the weights and the distance function $f : V(G) \rightarrow \N$ given by $f(v) = d(v,\bv_{\textrm{out}})$.  How well does this scheme classify novices and experts?
\item How well does the method-of-moments estimator we introduce above perform, in terms of bias or mean-squared error, for example?
\end{enumerate}

\section{Acknowledgments}

Thank you to David Feldon for valuable discussions and for introducing the author to cognitive task analysis.  

\end{document}